\documentclass[11pt,reqno]{amsart}
\usepackage{amsfonts, amssymb, amsmath, enumerate, bm, xspace}
\usepackage[linktoc=page, colorlinks, linkcolor=blue, citecolor=blue]{hyperref} 
\usepackage{pbox}

\numberwithin{equation}{section}

\DeclareMathOperator{\E}{\mathbb{E}}
\DeclareMathOperator{\Var}{Var}
\DeclareMathOperator*{\diag}{diag}

\DeclareMathOperator*{\tr}{tr}

\renewcommand{\Pr}[2][]{\mathbb{P}_{#1} \left\{ #2 \rule{0mm}{3mm}\right\}}
\newcommand{\ip}[2]{\langle#1,#2\rangle}

\def \R {\mathbb{R}}

\def \CC {\mathcal{C}}

\def \MM {\mathcal{M}}

\def \e {\varepsilon}

\def \l {\lambda}
\def \s {\sigma}

\def \tran {\mathsf{T}}

\def \one {{\textbf I}}
\def \onevector {{\bf 1}}

\def \pbar {\bar{p}}
\def \xbar {\bar{x}}
\def \Abar {\bar{A}}
\def \Zbar {\bar{Z}}
\def \Pbar {\bar{P}}
\def \xhat {\widehat{x}}

\def \Zhat {\widehat{Z}}
\def \Phat {\widehat{P}}
\def \Gr {\mathrm{G}}
\def \Mgr {\MM_\Gr}
\def \Mopt {\MM_{\mathrm{opt}}}
\def \In {\mathrm{In}}
\def \Out {\mathrm{Out}}

\newtheorem{theorem}{Theorem}[section]

\newtheorem{corollary}[theorem]{Corollary}
\newtheorem{lemma}[theorem]{Lemma}

\newtheorem{fact}[theorem]{Fact}

\theoremstyle{remark}
\newtheorem{remark}[theorem]{Remark}

\title[]{Community detection in sparse networks via Grothendieck's inequality}

\author{Olivier Gu\'edon \and Roman Vershynin}

\address{Universit\'{e} Paris-Est,
  Laboratoire d'Analyse et Math\'{e}matiques Appliqu\'ees (UMR 8050),
  UPEMLV, F-77454, Marne-la-Vall\'ee Cedex 2, France}

\address{Department of Mathematics, University of Michigan, 530 Church St., Ann Arbor, MI 48109, U.S.A.}

\thanks{O. G. is supported by the ANR project GeMeCoD, ANR 2011 BS01 007 01.
  R. V. is partially supported by NSF grants 1265782 and U.S. Air Force grant FA9550-14-1-0009.}

\date{\today}

\begin{document}

\begin{abstract}
  We present a simple and flexible method to prove consistency of semidefinite 
  optimization problems on random graphs.  
  The method is based on Grothendieck's inequality.
  Unlike the previous uses of this inequality that lead to constant relative accuracy, 
  we achieve any given relative accuracy by leveraging randomness.
  We illustrate the method with the problem of community detection in sparse networks,
  those with bounded average degrees.
  We demonstrate that even in this regime, various simple and natural semidefinite programs 
  can be used to recover the community structure up to an arbitrarily small fraction 
  of misclassified vertices. The method is general; it can be applied to a variety of
  stochastic models of networks and semidefinite programs.
\end{abstract}

\maketitle

\section{Introduction}

\subsection{Semidefinite problems on random graphs}				\label{s: semidefinite}

In this paper we present a simple and general method to prove consistency 
of various semidefinite optimization problems on random graphs.

Suppose we observe one instance of an $n \times n$ symmetric random matrix $A$ 
with unknown expectation $\bar{A} := \E A$. We would like to estimate the solution of
the discrete optimization problem
\begin{equation}							\label{eq: opt x Abar x}
\text{maximize } x^\tran \bar{A} x 
\quad \text{subject to} \quad
x \in \{-1,1\}^n.
\end{equation}
A motivating example of $A$ is the adjacency matrix of a random graph;
the Boolean vector $x$ can represent a partition of vertices of the graph into two classes.
Such Boolean problems can be encountered in the context of 
{\em community detection in networks} which we will discuss shortly. 
For now, let us keep working with the general class of problems \eqref{eq: opt x Abar x}.

Since $\bar{A}$ is unknown, one might hope to estimate 
the solution $\xbar$ of \eqref{eq: opt x Abar x} by 
solving the random instance of this problem, that is 
\begin{equation}							\label{eq: opt x A x}
\text{maximize } x^\tran A x 
\quad \text{subject to} \quad
x \in \{-1,1\}^n.
\end{equation}
The integer quadratic problem \eqref{eq: opt x A x} 
is NP-hard for general (non-random) matrices $A$. 
Semidefinite relaxations of many problems of this type have been proposed; see 
\cite{Goemans-Williamson, Lovasz-Schrijver, Nesterov, Alon-Naor} and the references therein.
Such relaxations are known to have {\em constant} relative accuracy.
For example, a semidefinite relaxation in \cite{Alon-Naor} computes, 
for any given positive semidefinite matrix $A$, a vector $x_0 \in \{-1,1\}^n$ such that
$x_0^\tran A x_0 \ge 0.56 \, \max_{x \in \{-1,1\}^n} x^\tran A x$.

In this paper we demonstrate how semidefinite relaxations of \eqref{eq: opt x A x} 
can recover a solution of \eqref{eq: opt x Abar x} with {\em any given relative accuracy}. 
Like several previously known methods, 
our approach is based on Grothendieck's inequality. We refer the reader to 
the surveys \cite{Pisier, Khot-Naor} for many reformulations and applications of this 
inequality in mathematics, computer science, optimization and other fields.
In contrast to the previous methods, 
we are going to apply Grothendieck's inequality for the (random) error $A - \bar{A}$ rather 
that the original matrix $A$, and this will be responsible for the arbitrary accuracy.

We will describe the general method in Section~\ref{s: method}. 
It is simple and flexible, and it can be used for showing consistency 
of a variety of semidefinite programs, which may or may not be related 
to Boolean problems like \eqref{eq: opt x Abar x}. 
But before describing the method, we would like to pause and give some 
concrete examples of results it yields for  {\em community detection}.

For simplicity, we will first focus on the classical stochastic block model, 
which is a random network whose nodes are split into two equal-sized clusters.
In Section~\ref{s: block model general} we will extend our discussion 
for broader models of networks almost without extra effort.

\subsection{Community detection: the classical stochastic block model}			

It is now customary to model networks as inhomogeneous random graphs \cite{Bollobas-Janson-Riordan}, 
which generalize the classical Erd\"os-R\'enyi model $G(n,p)$.
A benchmark example is the stochastic block model \cite{Holland-Laskey-Leinhardt}.
In this section we focus on the basic model with two communities of equal sizes; 
in Section~\ref{s: block model general} we will consider a more general situation.

We define a random graph on vertices $\{1,\ldots,n\}$ as follows.
Partition the set of vertices into two communities $\CC_1$ and $\CC_2$ 
of size $n/2$ each. 
For each pair of distinct vertices, 
we draw an edge independently with probability $p$ if both vertices belong 
to the same community, and $q$ (with $q \le p$) if they belong to different communities.   
For convenience we include the loops, so each vertex
has an edge connecting it to itself with probability $1$. 
This defines a distribution on random graphs which is denoted $G(n,p,q)$
and called the (classical) {\em stochastic block model}.
When $p=q$, we recover the classical Erd\"os-R\'enyi model of random graphs $G(n,p)$. 

The {\em community detection problem} asks to recover the communities $\CC_1$ and $\CC_2$
by observing one instance of a random graph drawn from $G(n,p,q)$.
As we will discuss in detail in Section~\ref{s: previous}, an array of algorithms is known to succeed for this 
problem for {\em relatively dense graphs}, those whose expected average degree 
(which is of order $pn$) is $\Omega(\log n)$, while less is known for {\em totally sparse graphs} -- 
those with bounded average degrees, i.e. with $pn = O(1)$. 
Our paper focuses on this sparse regime. 

\medskip

Recovery of the communities $\CC_1$ and $\CC_2$  
is equivalent to estimating the community membership vector, which we can define as
\begin{equation}							\label{eq: xbar}
\xbar \in \{-1,1\}^n, \quad
\xbar_i = 
\begin{cases}
  \phantom{-}1, & i \in \CC_1 \\
  -1, & i \in \CC_2. 
\end{cases}
\end{equation}
We will estimate $\bar{x}$ using 
the following semidefinite optimization problem:
\begin{equation}			\label{eq: SDP lambda}
\begin{aligned}
  &\text{maximize } \ip{A}{Z} - \l \ip{E_n}{Z} \\
  &\text{subject to } Z \succeq 0, \; \diag(Z) \preceq \one_n.
\end{aligned}
\end{equation}
Here  the inner product of matrices is defined in the usual way, 
that is $\ip{A}{B} = \tr(AB) = \sum_{i,j} A_{ij} B_{ij}$,  $\one_n$ denotes the identity matrix, 
the matrix $E_n$ has all entries equal $1$, 
and $A \succeq B$ means that $A-B$ is positive semidefinite. Observe that $E_n = \onevector_n \onevector_n^\tran$ 
where $\onevector_n \in \R^n$ is the vector whose all coordinates equal $1$. The 
constraint $\diag(Z) \preceq \one_n$ in \eqref{eq: SDP lambda} simply means that 
all diagonal entries of $Z$ are bounded by $1$.

For the value of $\l$ we choose the average degree of the graph (with loops removed), 
which is 
\begin{equation}							\label{eq: lambda}
\l = \frac{2}{n(n-1)} \sum_{i < j} a_{ij}
\end{equation}
where $a_{ij} \in \{0,1\}$ denote the entries of the adjacency matrix $A$.

\begin{theorem}[Community detection in classical stochastic block model]		\label{thm: community detection}
  Let $\e \in (0,1)$ and $n \ge 10^4 \e^{-2}$. 
  Let $A$ be the adjacency matrix of the random graph drawn from the stochastic 
  block model $G(n,p,q)$ with $\max\{p(1-p), q(1-q)\} \ge \frac{20}{n}$.
  Assume that $p=\frac{a}{n} > q=\frac{b}{n}$, and
  \begin{equation}							\label{eq: ab}
  (a-b)^2 \ge 1 0^4 \, \e^{-2} (a+b).
  \end{equation}
  Let $\Zhat$ be a solution of the semidefinite program \eqref{eq: SDP lambda}.
  Then, with probability at least $1-  e^3 5^{-n}$, we have
  \begin{equation}							\label{eq: Zhat error}
  \|\Zhat - \xbar \xbar^\tran\|_2^2 \le \e n^2 = \e \| \xbar \xbar^\tran \|_2^2.
  \end{equation}
\end{theorem}
Here and in the rest of this paper, $\|\cdot\|_2$ denotes the Frobenius norm of matrices
and the Euclidean norm of vectors. 
\medskip

Once we have estimated the rank-one matrix $\bar{x} \bar{x}^\tran$ using Theorem~\ref{thm: community detection}, 
we can also estimate the community membership vector $\bar{x}$ itself 
in a standard way, namely by computing the leading eigenvector. 

\begin{corollary}[Community detection with $o(n)$ misclassified vertices]		\label{cor: vector recovery}
  In the setting of Theorem~\ref{thm: community detection}, 
  let $\xhat$ denote an eigenvector of $\Zhat$ 
  corresponding to the largest eigenvalue, and with $\|\xhat\|_2 = \sqrt n$. Then 
  $$
  \min_{\alpha = \pm 1} \|\alpha \xhat - \xbar\|_2^2 \le \e n = \e \|\xbar\|_2^2.
  $$
  In particular, the signs of the coefficients of $\xhat$ correctly estimate the 
  partition of the vertices into the two communities, up to at most $\e n$ misclassified vertices. 
\end{corollary}


As we will discuss in Section~\ref{s: totally sparse} in more detail, there are previously known
algorithms for recovery of two communities under conditions similar to \eqref{eq: ab}.
These include a spectral clustering algorithm based on truncating the high degree vertices (whose analysis 
can be derived from \cite{FKS, Feige-Ofek}), combinatorial algorithms of \cite{Massoulie, Mossel-Neeman-Sly} 
based on path counting, and an algorithm \cite{Mossel-Neeman-Sly-belief} based on belief propagation, 
which minimizes the fraction of misclassified vertices.

An array of simple semidefinite programs like \eqref{eq: SDP lambda} and \eqref{eq: Zbar membership} 
has been proposed in networks community. Such programs have been analyzed for relatively dense graphs;
see \cite{Amini-Levina} for a review. It has been unknown if they could succeed for totally sparse graphs, where the expected degree 
is of constant order. 
Theorem~\ref{thm: community detection} provides a positive answer to this question. 
Moreover, the method of this paper is flexible enough to analyze many semidefinite programs, and it
can be applied for more general models of sparse networks than any previous results.

To illustrate this point, we will now choose a different semidefinite program and show that it succeeds 
for a large class of stochastic models of networks.
Moreover, in Section~\ref{s: balanced planted partition} 
we will see that a minor modification of the semidefinite program \eqref{eq: SDP lambda} 
also works well for multiple communities of equal sizes.

\subsection{Community detection: general stochastic block models}		\label{s: block model general}

Let us describe a model of networks where one can have
multiple communities of arbitrary sizes, arbitrarily many outliers, 
and unequal edge probabilities. 

To define such {\em general stochastic block model}, we assume that the set of 
vertices $\{1,\ldots,n\}$ is partitioned into communities
$\CC_1,\ldots,\CC_K$ of arbitrary sizes. We do not restrict the 
sizes of the communities, so in particular this model can automatically handle outliers,
the vertices that form communities of size $1$. 
For each pair of distinct vertices $(i,j)$, we draw an edge 
between $i$ and $j$ independently and with certain fixed probability $p_{ij}$. 
For convenience we include the loops like in  the classical stochastic block model, so $p_{ii} = 1$. 
To promote more edges within than across the communities, 
we assume that there exist numbers $p > q$ (thresholds) such that 
\begin{equation}         \label{eq: thresholds}
\begin{aligned}
&p_{ij} \ge p \quad \text{if $i$ and $j$ belong to the same community};\\
&p_{ij} \le q \quad \text{if $i$ and $j$ belong to different communities}.
\end{aligned}
\end{equation}
The community structure of such a network is captured by
the cluster matrix matrix $\bar{Z} \in \{0,1\}^{n \times n}$ defined as
\begin{equation}							\label{eq: Zbar membership}
\bar{Z}_{ij} = 
\begin{cases}
  1 & \text{if $i$ and $j$ belong to the same community}; \\
  0 & \text{if $i$ and $j$ belong to different communities}.
\end{cases}
\end{equation}
We will estimate $\bar{Z}$ using 
the following semidefinite optimization program:
\begin{equation}							\label{eq: SDP sum fixed}
\begin{aligned}
  &\text{maximize } \ip{A}{Z} \\
  &\text{subject to } Z \succeq 0, \; Z \ge 0, \; \diag(Z) \preceq \one_n, \; \textstyle{\sum_{i,j=1}^n Z_{ij} = \l}.
\end{aligned}
\end{equation}
Here as usual $Z \succeq 0$ means that $Z$ is positive semidefinite, 
and $Z \ge 0$ means that all entries of $Z$ are non-negative.
We choose the value of $\lambda$ to be the number of elements
in the cluster matrix, that is 
\begin{equation}         \label{eq: lambda general}
\l = \sum_{i,j=1}^n \bar{Z}_{ij} = \sum_{k=1}^K |\CC_k|^2.
\end{equation}
If all communities have the same size $s$, then $\lambda = K s^2 = n s$.

\begin{theorem}[Community detection in general stochastic block model]    \label{thm: community detection general}
  Let $\e \in (0,1)$.
  Let $A$ be the adjacency matrix of the random graph drawn from the general stochastic 
  block model described above. Denote by $\pbar$ the expected variance of the edges, that is 
  $\pbar = \frac{2}{n(n-1)} \sum_{i<j} p_{ij} (1 - p_{ij})$.
  Assume that $p=\frac{a}{n} > q=\frac{b}{n}$, $\pbar = \frac{g}{n}$, $g \ge 9$ and  
  \begin{equation}							\label{eq: ab general}
  (a-b)^2 \ge 484\, \e^{-2} g.
  \end{equation}
  Let $\Zhat$ be a solution of the semidefinite program \eqref{eq: SDP sum fixed}.
  Then, with probability at least $1-e^3 5^{-n}$, we have
  \begin{equation}							\label{eq: Zhat error general}
  \|\Zhat - \bar{Z}\|_2^2 \le \|\Zhat - \bar{Z}\|_1 
  \le \e n^2.
  \end{equation}
\end{theorem}
Here as usual $\|\cdot\|_2$ denotes the Frobenius norm of matrices, and $\|\cdot\|_1$ 
denotes the $\ell_1$ norm of the matrices considered as vectors, that is 
$\|(a_{ij})\|_1 = \sum_{i,j} |a_{ij}|$.

\begin{remark}[General community structure]
  The power of Theorem~\ref{thm: community detection general} does not depend 
  on the community structure, i.e. on the number and sizes of the communities.
  This seemingly surprising observation can be explained by the fact that small 
  communities, those with sizes $o(n)$, can get absorbed in the error term in \eqref{eq: Zhat error general}, 
  so they will not be recovered. 
\end{remark}

\begin{remark}[If the sizes of communities are not known]		\label{rem: generality of community structure}
  Our choice of the parameter $\lambda$ in \eqref{eq: lambda general} assumes that 
  we know the sizes of the communities. What if they are not known?
  From the proof of Theorem~\ref{thm: community detection general} it 
  will be clear what happens when $\lambda>0$ is chosen {\em arbitrarily}.
   Assume that we choose $\l$ so that 
  $\l \le \l_0 := \sum_k |\CC_k|^2$. Then instead of estimating the full cluster 
  graph (described in Remark~\ref{rem: cluster graph}), 
  the solution $\Zhat$ will only estimate a certain {\em subgraph} of the cluster graph, 
  which may miss at most $\l_0 - \l$ edges. 
  On the other hand, if we choose $\l$ so that 
  $\l \ge \l_0$, then the solution $\Zhat$ will estimate a certain {\em supergraph} 
  of the cluster graph, which may have at most $\l - \l_0$ extra edges. In either case, 
  such solution could be meaningful in practice. 
\end{remark}

\begin{remark}[Cluster graph]						\label{rem: cluster graph}
  It may be convenient to view the cluster matrix $\bar{Z}$ 
  as the adjacency matrix of the {\em cluster graph}, in which all vertices within each 
  community are connected and there are no connections across the communities. 
  This way, the semidefinite program \eqref{eq: SDP sum fixed} takes a sparse graph as an input,
  and it returns an estimate of the cluster graph as an output. The effect of the program is thus
  to ``densify'' the network inside the communities and ``sparsify'' it across the communities.
\end{remark}

\begin{remark}[Other semidefinite programs]
  There is nothing special about the semidefinite programs \eqref{eq: SDP lambda} 
  and \eqref{eq: SDP sum fixed}.
  For example, one can tighten the constraints and instead of 
  $\diag(Z) \preceq \one_n$ require that $\diag(Z) = \one_n$ in both programs. 
  Similarly, instead of placing in \eqref{eq: SDP sum fixed} the constraint on the sum of all entries of $Z$,
  one can place constraints on the sums of each row.
  In a similar fashion, one should be able to analyze other semidefinite relaxations, 
  both new and those proposed in the previous literature on community detection, see \cite{Amini-Levina}.
\end{remark}

For one more illustration for the method described here, we refer the reader to Section~7 of the extended version of this paper \cite{GVarxiv}.
There we consider a minor modification of the semidefinite program \eqref{eq: SDP lambda},
and we show that it succeeds in presence of multiple communities of equal sizes 
(the so-called {\em balanced planted partition model}). The sufficient condition for that is 
$(a-b)^2 \ge 50^2 \e^{-2} (a + b(K-1))$ where $K$ is the number of communities, $s$ the size of the communities and $p=a/s$, $q=b/s$.

\subsection{Related work}			\label{s: previous}

Community detection in stochastic block models is a fundamental problem that has 
been extensively studied in theoretical computer science and statistics. 
A plethora or algorithmic approaches have been proposed, in particular
those based on combinatorial techniques \cite{Bui-etal, Dyer-Frieze}, 
spectral clustering \cite{Boppana, Alon-Kahale, Alon, McSherry, Newman, Rohe-etal, 
Chaudhuri-etal, Nadakuditi-Newman, Lei-Rinaldo, Qin-Rohe, Joseph-Yu}, 
likelihood maximization \cite{Snijders-Nowicki, Bickel-Chen, Amini-etal},
variational methods \cite{Airoldi-etal, Celisse-etal, Bickel-etal},
Markov chain Monte Carlo \cite{Snijders-Nowicki, Nowicki-Snijders},
belief propagation \cite{Decelle-etal},
and convex optimization including semidefinite programming
\cite{Jalali-etal, Ames-Vavasis, Oymak-Hassibi, Ailon-etal, Chen-etal, Chen-Jalali-etal, Chen-Xu, 
  Cai-Li, Amini-Levina, Bui-etal}.

\subsubsection{Relatively dense networks: average degrees are $\Omega(\log n)$}

Most known rigorous results on community detection are proved for relatively dense 
networks whose expected degrees go to infinity with $n$. 
If the degrees grow no slower than $\log n$, it may be possible 
to recover the community structure {\em perfectly}, without any misclassified vertices. 
A variety of community detection methods are known to succeed
in this regime, including those based on spectral clustering,
likelihood maximization and convex optimization mentioned above; 
see e.g. \cite{McSherry, Bui-etal} and the references therein.

The semidefinite programs \eqref{eq: SDP lambda} and \eqref{eq: SDP sum fixed} 
are similar to those proposed in the recent literature,
most notably in \cite{Chen-etal, Chen-Xu, Cai-Li, Amini-Levina, Bui-etal}. The semidefinite 
relaxations discussed in \cite{Chen-Xu, Cai-Li} can perfectly recover 
the community structure if $(a-b)^2 \ge C (a \log n+b)$ for a sufficiently large constant $C$;
see \cite{Amini-Levina} for a review of these results.

\subsubsection{Totally sparse networks: bounded average degrees}			\label{s: totally sparse}

The problem becomes more difficult for sparser networks, whose expected average
degrees grow to infinity arbitrarily slowly or even remain bounded in $n$. 
Although studying such networks is well motivated from the practical perspective
\cite{Leskovec-etal, Strogatz}, little has been known on the theoretical level.

If the degrees grow slower than $\log n$, it is impossible to correctly classify 
all vertices, since with high probability a positive fraction of the vertices will be isolated. 
Still, the fraction of isolated vertices tends to zero with $n$,
so we can hope to correctly classify a {\em majority} of the vertices in this regime.

The spectral method developed by J.~Kahn and E.~Szemeredi for random regular graphs \cite{FKS}
can be adapted for Erd\"os-R\'enyi random graphs \cite{Alon-Kahale, Feige-Ofek} and, more generally, for 
the stochastic block model $G(n, \frac{a}{n}, \frac{b}{n})$. 
If one {\em truncates the graph} by removing all vertices with too large degrees (say, larger than $10(a+b)$),
then the argument of \cite{FKS, Feige-Ofek} can be adapted to conclude that with some positive probability,
the truncated adjacency matrix concentrates near its expectation in the spectral norm. 
The communities can then be approximately recovered using the {\em spectral clustering}, 
which is based on the signs of the coefficients of the second eigenvector.
Working out the details, one finds that a sufficient condition for this method to succeed
is similar to \eqref{eq: ab}, that is 
\begin{equation}         \label{eq: optimal ab}
(a-b)^2 \ge C_\e (a+b) 
\end{equation}
where $C_\e$ depends only on the desired accuracy $\e$ or recovery.
However, for real networks it is usually  impractical to remove high degree vertices and the probabilistic estimate from \cite{Feige-Ofek} is not sharp. 

A.~Coja-Oghlan \cite{Coja-Oghlan} proposed a different, complicated adaptive spectral algorithm 
that can approximately recover communities under the condition $(a-b)^2 \ge C_\e (a+b) \log (a+b)$.
Recently, L. Massouli\'e \cite{Massoulie} and E.~Mossel, J.~Neeman and A.~Sly \cite{Mossel-Neeman-Sly} 
came up with combinatorial algorithms based on path counting,
which can approximately recover communities under the condition \eqref{eq: optimal ab}. 
These results are stated in the asymptotic regime for $n \to \infty$ and 
without explicit dependence of $C_\e$ on the desired accuracy $\e$.
Furthermore, E.~Mossel, J.~Neeman and A.~Sly developed an algorithm 
based on belief propagation \cite{Mossel-Neeman-Sly-belief}, 
which minimizes the fraction of misclassified vertices.

Condition \eqref{eq: optimal ab} has the optimal form. 
Indeed, it was shown in \cite{Mossel-Neeman-Sly-consistency} that the lower bound \eqref{eq: optimal ab} 
is required for any algorithm to be able to recover communities with at most $\e n$ misclassified vertices, 
where $C_\e \to \infty$ as $\e \to 0$.
A conjecture of A. Decelle, F. Krzakala, C. Moore and L. Zdeborova  proved 
recently by E.~Mossel, J.~Neeman and A.~Sly \cite{Mossel-Neeman-Sly lower, Mossel-Neeman-Sly} and 
Massouile \cite{Massoulie} states that one can find a partition {\em correlated} with the true community 
partition (i.e. with the fraction of misclassified vertices bounded away from $50\%$ as $n \to \infty$) 
if $(a-b)^2 \ge C(a+b)$ with some constant $C>2$. Moreover, this result achieves information-theoretic limit: 
no algorithm can succeed if $C \le 2$.

It remains an open question whether semidefinite programing can achieve similar information-theoretic
limits. Theorem~\ref{thm: community detection} does not achieve them; addressing this problem will 
require to tighten the absolute constant and the dependence on $\e$ in \eqref{eq: ab}.

\subsubsection{The new results in historical perspective}

A variety of simple semidefinite programs like \eqref{eq: SDP lambda} and \eqref{eq: Zbar membership} 
have been proposed in the network literature. Such programs have been analyzed only for dense networks
where the degrees grow as $\Omega(\log n)$ in which case perfect community detection is possible.
The present paper shows that the same semidefinite programs succeed for totally sparse networks as well,
producing a small number of misclassified vertices; moreover the sufficient condition \eqref{eq: optimal ab} 
is optimal up to an absolute constant. 

Furthermore, the method of the present paper generalizes smoothly 
for a broad classes of sparse networks. We saw in Section~\ref{s: block model general} 
that semidefinite programming succeeds for networks with {\em variable edge probabilities} $p_{ij}$;
community detection in such networks seems to be out of reach for known spectral methods.

We also saw how networks with multiple communities be handled with semidefinite 
approach. This has been studied in the statistical literature before; 
the semidefinite relaxations proposed in \cite{Chen-etal, Chen-Xu, Cai-Li, Amini-Levina}
were designed for multiple communities and outliers. However, previous theoretical results 
for multiple communities were only available for dense regime where the degrees grow as
 $\Omega(\log n)$, in which case perfect community detection is possible.

\subsubsection{Follow up work}

After this paper had been submitted, several new results appeared on community detection
in stochastic block models. We will mention here only results that apply for totally sparse networks. 
The initial discovery of \cite{Massoulie, Mossel-Neeman-Sly} mentioned in
Section~\ref{s: totally sparse} was followed by the work \cite{Bordenave-Lelarge-Massoulie}.
Semidefinite programs on random graphs were further analyzed in \cite{Montanari-Sen}
using higher-rank Grothendieck inequalities and insights from mathematical physics. 
Stochastic block models with labeled edges were addressed in \cite{LMX}
using truncated spectral clustering (with high degree vertices removed, based on \cite{Feige-Ofek}) 
and semidefinite programming (whose analysis is based on the method of the present paper). 
A two-stage algorithm based on truncated spectral clustering
and swapping vertices (like e.g. in \cite{Mossel-Neeman-Sly-consistency})
was analyzed in \cite{Chin-Rao-Vu}; the swapping stage 
leads to the sufficient condition \eqref{eq: optimal ab} with 
with an optimal dependence on the accuracy, $C_\e \sim \log(1/\e)$.
A different combinatorial method was proposed and analyzed in \cite{Abbe-Sandon};
regularized spectral clustering was shown to succeed in \cite{LLV, LV};
and a computationally feasible likelihood-based algorithm that minimizes 
the risk for misclassification proportion was found in \cite{Gao-Ma-Zhang-Zhou}.
Some of the mentioned work can be used for networks with multiple communities, 
see \cite{Chin-Rao-Vu, Abbe-Sandon, LLV, LV, Gao-Ma-Zhang-Zhou}.

\subsection{Plan of the paper}

We discuss the method in general terms in Section~\ref{s: method}. 
We explain how Grothendieck's inequality can be used to show 
tightness of various semidefinite programs on random graphs. 
Section~\ref{s: Grothendieck} is devoted to Grothendieck's inequality and its implications 
for semidefinite programming. 
In Section~\ref{s: deviation} we prove a simple concentration inequality
for random matrices in the cut norm. 
In Section~\ref{s: community detection} we specialize to the community detection problem 
for the classical stochastic block model, and we prove Theorem~\ref{thm: community detection}
and Corollary~\ref{cor: vector recovery} there. 
In Section~\ref{s: community detection general} we consider the general classical stochastic block model, 
and we prove Theorem~\ref{thm: community detection general} there. 

\subsection*{Acknowledgement}
This work was carried out while the first author was a Gerhing Visiting Professor 
at the University of Michigan. He thanks this institution for hospitality. 
The second author is grateful to Alexander Barvinok for drawing his attention to 
Y.~Nesterov's work \cite{Nesterov} on combinatorial optimization and to Grothendieck's inequality
in this context. We also thank Elchanan Mossel for useful discussions, 
and the anonymous referees whose suggestions helped to improve the presentation.

\section{Semidefinite optimization on random graphs: the method in a nutshell}   \label{s: method}

In this section we explain the general method of this paper, which can be applied 
to a variety of optimization problems. 
To be specific, let us return to the problem we described in Section \ref{s: semidefinite}, 
which is to estimate the solution $\xbar$ of the optimization problem \eqref{eq: opt x Abar x}
from a single observation of the random matrix $A$.
We suggested there to approximate $\bar{x}$ by the solution of the (random) 
program \eqref{eq: opt x A x}, which we can rewrite as follows:
\begin{equation}							\label{eq: opt A x xtran}
\text{maximize } \ip{A}{x x^\tran} 
\quad \text{subject to} \quad
x \in \{-1,1\}^n.
\end{equation}
Note that if we maximized $\ip{A}{x x^\tran}$ over the Euclidean ball $B(0,\sqrt{n})$, then the problem 
would be simple -- the solution $x$ would be the eigenvector corresponding to the eigenvalue of $A$ 
of largest magnitude. This simpler problem underlies the most basic algorithm
for community detection called {\em spectral clustering}, where the communities 
are recovered based on the signs of an eigenvector of the adjacency matrix 
(going back to \cite{Hagen-Kahng, Boppana, McSherry}, see \cite{Rohe-etal}). 
The optimization problem \eqref{eq: opt A x xtran} is harder and more subtle; the replacement 
of the Euclidean ball by the cube introduces  a strong restriction on the coordinates of $x$. 
This restruction rules out {\em localized} solutions $x$ where most of the mass of $x$ is concentrated 
on a small fraction of coordinates. 
Since eigenvectors of sparse matrices tend to be localized (see \cite{Bordenave-Guionnet}), 
basic spectral clustering is often unsuccessful for sparse networks.

\medskip

Let us choose a convex subset $\Mopt$ of the set of positive semidefinite matrices whose all entries are 
bounded by $1$ in absolute value. (For now, it can be any subset.) 
Note that $x x^\tran$ appearing in \eqref{eq: opt A x xtran} are 
examples of such matrices. 
We consider the following semidefinite relaxation of \eqref{eq: opt A x xtran}:
\begin{equation}							\label{eq: opt AZ}
\text{maximize } \ip{A}{Z} 
\quad \text{subject to} \quad
Z \in \Mopt.
\end{equation}
We might hope that the solution $\Zhat$ of this program would enable us to estimate 
the solution $\xbar$ of \eqref{eq: opt x Abar x}.

To realize this hope, one needs to check a few things, which may or may not be true
depending on the application. 
First, one needs to design the feasible set $\Mopt$ in such a way 
that {\em the semidefinite relaxation of the expected problem \eqref{eq: opt x Abar x} is tight}. This means that
the solution $\bar{Z}$ of the program 
\begin{equation}							\label{eq: opt AZbar}
\text{maximize } \ip{\bar{A}}{Z} 
\quad \text{subject to} \quad
Z \in \Mopt
\end{equation}
satisfies 
\begin{equation}							\label{eq: tightness of relaxation}
\bar{Z} = \bar{x} \bar{x}^\tran.
\end{equation}
This condition can be arranged for in various applications. In particular, 
this is the case in the setting of Theorem~\ref{thm: community detection}; we show this in Lemma~\ref{lem: Zbar}.

\medskip

Second, one needs a {\em uniform deviation inequality}, which would guarantee with high probability that
\begin{equation}							\label{eq: deviation}
\max_{x,y \in \{-1,1\}^n} |\ip{A-\bar{A}}{xy^\tran}| \le \e.
\end{equation}
This can often be proved by applying standard deviation inequalities for a fixed pair $(x,y)$, 
followed by a union bound over all such pairs. We prove such a deviation inequality 
in Section~\ref{s: deviation}.

\medskip

Now we make the crucial step, which is an application of {\em Grothendieck's inequality}. 
A reformulation of this remarkable inequality, which we explain in Section~\ref{s: Grothendieck}, states that 
\eqref{eq: deviation} automatically implies that 
\begin{equation}							\label{eq: deviation on Mopt}
\max_{Z \in \Mopt} |\ip{A-\bar{A}}{Z}| \le C\e.
\end{equation}
This will allow us to conclude that the solution $\Zhat$ of \eqref{eq: opt AZ} approximates 
the solution $\bar{Z}$ of \eqref{eq: opt AZbar}. To see this, let us compare the value of the expected 
objective function $\ip{\bar{A}}{Z}$ at these two vectors. We have
\begin{align}
\ip{\bar{A}}{\Zhat} 
&\ge \ip{A}{\Zhat} - C\e			\quad \text{(replacing $\bar{A}$ by $A$ using \eqref{eq: deviation on Mopt})} \nonumber\\
&\ge \ip{A}{\bar{Z}} - C\e			\quad \text{(since $\Zhat$ is the maximizer in \eqref{eq: opt AZ})} \nonumber\\
&\ge \ip{\bar{A}}{\bar{Z}} - 2C\e			
  \quad \text{(replacing $A$ by $\bar{A}$ back using \eqref{eq: deviation on Mopt}).}   \label{eq: almost maximizer}
\end{align}
This means that $\Zhat$ almost maximizes the objective function $\ip{\bar{A}}{Z}$ in \eqref{eq: opt AZbar}.

\medskip

The final piece of information we require is that the expected objective function $\ip{\bar{A}}{Z}$
{\em distinguishes points near its maximizer} $\bar{Z}$. This would allow one to 
automatically conclude from \eqref{eq: almost maximizer} 
that the almost maximizer $\Zhat$ is close to the true maximizer, i.e. that 
\begin{equation}							\label{eq: Zhat-Zbar small}
\|\Zhat - \bar{Z}\| \le \text{something small}
\end{equation}
where $\|\cdot\|$ can be the Frobenius or operator norm. 
Intuitively, the requirement that the objective function distinguishes points 
amounts to a non-trivial curvature of the feasible set $\Mopt$ at the maximizer $\bar{Z}$.
In many situations, this property is easy to verify. In the setting of 
Theorems~\ref{thm: community detection} and \ref{thm: community detection general}, 
we check it in Lemma~\ref{lem: Zhat-Zbar} and Lemmas~\ref{lem: distinguishes points}--\ref{lem: Zhat-Zbar sum fixed}
respectively.

\medskip

Finally, we can recall from \eqref{eq: tightness of relaxation} that $\bar{Z} = \bar{x} \bar{x}^\tran$. 
Together with \eqref{eq: Zhat-Zbar small}, this yields that $\Zhat$ is approximately a rank-one matrix, 
and its leading eigenvector $\xhat$ satisfies 
$$
\|\xhat - \bar{x}\|_2 \le  \text{something small}.
$$
Thus we estimated the solution $\bar{x}$ of the problem \eqref{eq: opt x Abar x}
as desired.

\begin{remark}[General semidefinite programs]
  For this method to work, it is not crucial that the semidefinite program be a relaxation of 
  any vector optimization problem. Indeed, one can analyze semidefinite programs 
  of the type \eqref{eq: opt AZ} without any vector optimization problem \eqref{eq: opt A x xtran}
  in the background. In such cases, the requirement \eqref{eq: tightness of relaxation}
  of tightness of relaxation can be dropped. The solution $\bar{Z}$ may itself be informative. 
  An example of such situation is Theorem~\ref{thm: community detection general}
  where the community membership matrix $\bar{Z}$ is important by itself. However, $\bar{Z}$
  can not be represented as $\bar{x} \bar{x}^\tran$ for any $\bar{x}$, since $\bar{Z}$ is not a rank one matrix.  
\end{remark}

\section{Grothendieck's inequality and semidefinite programming}					\label{s: Grothendieck}

Grothendieck's inequality is a remarkable result proved originally in the functional 
analytic context \cite{Grothendieck} and reformulated in \cite{Lindenstrauss-Pelczynski} 
in the form we are going to describe below. This inequality had found
applications in several areas \cite{Pisier, Khot-Naor}.
It has already been used to analyze semidefinite relaxations
of hard combinatorial optimization problems \cite{Nesterov, Alon-Naor}, although previous 
relaxations lead to constant (rather than arbitrary) accuracy.

\begin{theorem}[Grothendieck's inequality]				\label{eq: grothendieck}
  Consider an $n \times n$ matrix of real numbers $B = (b_{ij})$. 
  Assume that 
  $$
  \Big| \sum_{i,j} b_{ij} s_i t_j \Big| \le 1
  $$
  for all numbers $s_i, t_i \in \{-1,1\}$.
  Then 
  $$
  \Big| \sum_{i,j} b_{ij} \ip{X_i}{Y_j} \Big| \le K_\Gr
  $$
 for all vectors $X_i,Y_i \in B_2^n$.  
\end{theorem}
Here $B_2^n = \{ x \in \R^n : \|x\|_2 \le 1\}$ is the unit ball for the Euclidean norm, 
and $K_\Gr$ is an absolute constant referred to as {\em Grothendieck's constant}.
The best value of $K_\Gr$ is still unknown, and the best known bound \cite{Naor-Twins} is
\begin{equation}         \label{eq: Grothendieck constant}
K_\Gr < \frac{\pi}{2 \ln(1+\sqrt{2})} \le 1.783.
\end{equation}

\subsection{Grothendieck's inequality in matrix form}

To restate Grothendieck's inequality in a matrix form, 
let us assume for simplicity that $m=n$ and observe that 
$\sum_{i,j} b_{ij} s_i t_j = \ip{B}{s t^\tran}$
where $s$ and $t$ are the vectors in $\R^n$ with coordinates $s_i$ and $t_j$ respectively.
Similarly, 
$\sum_{i,j} b_{ij} \ip{X_i}{Y_j} = \ip{B}{X Y^\tran}$
where $X$ and $Y$ are the $n \times n$ matrices with rows $X_i^\tran$ and $Y_j^\tran$ respectively.
This motivates us to consider the following two sets of matrices:
$$
\MM_1 := \left\{ s t^\tran :\; s, t \in \{-1,1\}^n \right\}, 
\quad
\Mgr := \left\{ XY^\tran :\; \text{all rows } X_i, Y_j \in B_2^n \right\}.
$$
Clearly, $\MM_1 \subset \Mgr$.
Grothendieck's inequality can be stated as follows:
  \begin{equation}							\label{eq: grothendieck matrix}
\forall B \in \R^{n \times n}, \quad  \max_{Z \in \Mgr} \left| \ip{B}{Z} \right| 
  \le K_\Gr \max_{Z \in \MM_1} \left| \ip{B}{Z} \right|.
  \end{equation}

We can view this inequality as a relation between two matrix norms. 
The right side of \eqref{eq: grothendieck matrix} defines
the $\ell_\infty \to \ell_1$ norm of $B = (b_{ij})$, which is
\begin{align}
\|B\|_{\infty \to 1} 
&= \max_{\|s\|_\infty \le 1} \|Bs\|_1
= \max_{s,t \in \{-1,1\}^n} \ip{B}{s t^\tran}
= \max_{s,t \in \{-1,1\}^n} \sum_{i,j=1}^n b_{ij} s_i t_j \nonumber\\
&= \max_{Z \in \MM_1} \left| \ip{B}{Z} \right|.			\label{eq: infinity to one}
\end{align}
We note in passing that this norm is equivalent to the so-called {\em cut norm}, 
whose importance in algorithmic problems is well understood in theoretical computer science community, 
see e.g. \cite{Alon-Naor, Khot-Naor}.

\medskip

Let us restrict our attention to the part of Grothendieck's set $\Mgr$ consisting of positive 
semidefinite matrices. To do so, we consider the following set of $n \times n$ matrices:
\begin{equation}							\label{eq: Mgr+}
\Mgr^+ := \left\{ Z:\; Z \succeq 0, \; \diag(Z) \preceq \one_n \right\}  \subset \Mgr \subset [-1,1]^{n \times n}.
\end{equation}
To check the first inclusion in \eqref{eq: Mgr+}, let $Z \in \Mgr^+$. 
Since $Z \succeq 0$, there exists a matrix $X$ such that $Z = X^2$. 
The rows $X_i^\tran$ of $X$ satisfy 
$
\|X_i\|_2^2 = \ip{X_i}{X_i} = (X^\tran X)_{ii} = Z_{ii} \le 1,
$
where the last inequality follows from the assumption $\diag(Z) \preceq \one_n$.
Choosing $Y=X$ in the definition of $\Mgr$, we conclude that $Z \in \Mgr$. 
To check the second inclusion in \eqref{eq: Mgr+}, note that for every matrix $X Y^\tran \in \Mgr$, 
we have $(XY^\tran)_{ij} = \ip{X_i}{Y_j} \le \|X_i\|_2 \; \|Y_j\|_2 \le 1.$

Combining \eqref{eq: grothendieck matrix} with \eqref{eq: Mgr+}
and the identity \eqref{eq: infinity to one},
we obtain the following form of Grothendieck inequality for positive semidefinite matrices.

\begin{fact}[Grothendieck's inequality, PSD]				\label{thm: grothendieck PSD}
  Every matrix $B \in \R^{n \times n}$ satisfies 
  $$
  \max_{Z \in \Mgr^+} \left| \ip{B}{Z} \right| 
  \le K_\Gr \, \|B\|_{\infty \to 1}.
  $$
\end{fact}

\subsection{Semidefinite programming}				\label{s: SDP on Grothendieck}

To keep the discussion sufficiently general, let us consider the following 
class of optimization programs:
\begin{equation}							\label{eq: SDP general}
  \text{maximize } \ip{B}{Z}
  \quad \text{subject to} \quad Z \in \Mopt.
\end{equation}
Here $\Mopt$ can be any subset of the Grothendieck's set $\Mgr^+$ defined in \eqref{eq: Mgr+}.
A good example is where $B$ is the adjacency matrix of a random graph, 
possibly dilated by a constant matrix.
For example, the semidefinite program \eqref{eq: SDP lambda} is of the form 
\eqref{eq: SDP general} with $\Mopt = \MM_\Gr^+$
and $B = A - \l E_n$.

Imagine that there is a similar but simpler problem where $B$ is replaced by a certain {\em reference matrix} $R$, 
that is 
\begin{equation}							\label{eq: SDP reference}
  \text{maximize } \ip{R}{Z}
  \quad \text{subject to} \quad Z \in \Mopt.
\end{equation}
A good example is where $B$ is a random matrix and $R = \E B$; this will be the case in
the proof of Theorem~\ref{thm: community detection general}.
Let $\Zhat$ and $Z_R$ be the solutions of the original problem \eqref{eq: SDP general}
and the reference problem \eqref{eq: SDP reference} respectively, thus 
$$
\Zhat := \arg \max_{Z \in \Mopt} \ip{B}{Z}, \quad 
Z_R := \arg \max_{Z \in \Mopt} \ip{R}{Z}.
$$
The next lemma shows that $\Zhat$ provides an almost optimal solution to 
the reference problem if the original and reference matrices $B$ and $R$ are close.

\begin{lemma}[$\Zhat$ almost maximizes the reference objective function]		\label{lem: almost max}
We have
  \begin{equation}							\label{eq: almost max}
  \ip{R}{Z_R} - 2 K_\Gr \|B-R\|_{\infty \to 1}
  \le \ip{R}{\Zhat} 
  \le \ip{R}{Z_R}.
  \end{equation}
\end{lemma}

\begin{proof}
The upper bound is trivial by definition of $Z_R$.
The lower bound is based on Fact~\ref{thm: grothendieck PSD}, 
which implies that for every $Z \in \Mopt$, one has
\begin{equation}							\label{eq: B-Bbar}
|\ip{B-R}{Z}| \le  K_\Gr \|B-R\|_{\infty \to 1} =: \e.
\end{equation}
Now, to prove the lower bound in \eqref{eq: almost max}, we 
will first replace $R$ by $B$ using \eqref{eq: B-Bbar},
then replace $\Zhat$ by $Z_R$ using the fact that $\Zhat$ is a maximizer for $\ip{B}{Z}$, 
and finally replace back $B$ by $R$ using \eqref{eq: B-Bbar} again.
This way we obtain 
$$
\ip{R}{\Zhat}
\ge \ip{B}{\Zhat} - \e 
\ge \ip{B}{Z_R} - \e 
\ge \ip{R}{Z_R} - 2\e.
$$
This completes the proof of Lemma~\ref{lem: almost max}.
\end{proof}

\section{Deviation in the cut norm}					\label{s: deviation}

To be able to effectively use Lemma~\ref{lem: almost max}, 
we will now show how to bound the cut norm of random matrices. 

\begin{lemma}[Deviation in $\ell_\infty \to \ell_1$ norm]				\label{lem: deviation}
  Let $A = (a_{ij}) \in \R^{n \times n}$ be a symmetric matrix whose diagonal entries equal  1, whose entries
 above the diagonal are independent random variables satisfying $0 \le a_{ij}  \le 1$.
  Assume that 
  \begin{equation}							\label{eq: pbar}
  \pbar := \frac{2}{n(n-1)} \sum_{i<j} \Var(a_{ij})\ge \frac{9}{n}.
  \end{equation}
  Then, with probability at least $1-  e^3 5^{-n}$, we have 
  $$
  \|A - \E A\|_{\infty \to 1} \le 3 \, \pbar^{1/2} n^{3/2}.
  $$
\end{lemma}

We will shortly deduce Lemma~\ref{lem: deviation} from Bernstein's inequality 
followed by a union bound over $x,y \in \{-1,1\}^n$; arguments of this type are 
standard in the analysis of random graphs (see e.g. \cite[Section~2.3]{Bollobas}).
But before we do this, let us 
pause to explain the conclusion of Lemma~\ref{lem: deviation}.

\begin{remark}[Regularization effect of $\ell_\infty \to \ell_1$ norm]
  Let us test Lemma~\ref{lem: deviation} on the simple example 
  where $A$ is the adjacency matrix of a sparse Erd\"os-Renyi random graph $G(n,p)$ 
  with $p=a/n$, $a \ge 1$. Here we have $\pbar = p(1-p) \le p = a/n$. Lemma~\ref{lem: deviation} states that 
  $\|A - \E A\|_{\infty \to 1} \le 3 a^{1/2} n$. This can be compared with 
  $\|\E A\|_{\infty \to 1} = (1 + p(n-1)) n \ge an$. So we obtain
  $$
  \|A - \E A\|_{\infty \to 1} \le 3 a^{-1/2} \, \|\E A\|_{\infty \to 1}.
  $$
  This deviation inequality is good when $a$ exceeds a sufficiently large
  absolute constant. Since that $a = pn$ is the expected average degree of the graph, 
  it follows that we can handle graphs with {\em bounded expected degrees}. 
  
  \medskip
  
  This is a good place to note the importance of the $\ell_\infty \to \ell_1$ norm.
  Indeed, for the {\em spectral norm} a similar concentration inequality would fail.   
  As is well known and easy to check, for $a =O(1)$ one would have $\|A - \E A\| \gg \|\E A\|$
  due to contributions from high degree vertices. In fact, those are the only obstructions to concentration. 
  Indeed, according to a result of U.~Feige and E.~Ofek \cite{Feige-Ofek}, the removal of high-degree vertices forces 
  a non-trivial concentration inequality to hold in the spectral norm. 
  In contrast to this, the $\ell_\infty \to \ell_1$ norm does not feel the vertices 
  with high degrees. It has an automatic {\em regularization effect}, which averages the 
  contributions of all vertices, and in particular the few high degree vertices. 
\end{remark}

The proof of Lemma~\ref{lem: deviation} will be based on Bernstein's inequality, which we quote here
(see, for example, Theorem 1.2.6 in \cite{lamabook}).

\begin{theorem}[Bernstein's inequality]				\label{thm: bernstein}
  Let $Y_1,\ldots,Y_N$ be independent random variables such that $\E Y_k = 0$ and $|Y_k| \le M$. 
  Denote $\s^2 = \frac{1}{N} \sum_{k=1}^N \Var(Y_k)$.
  Then for any $t \ge 0$, one has
  $$
  \Pr{ \frac{1}{N} \sum_{k=1}^N Y_k > t }  
  \le \exp \left( - \frac{N t^2/2}{\s^2 + Mt/3} \right).
  $$
\end{theorem}

\begin{proof}[Proof of Lemma~\ref{lem: deviation}]
Recalling the definition \eqref{eq: infinity to one} of the 
$\ell_\infty \to \ell_1$ norm, we see that we need to bound 
\begin{equation}         \label{eq: A-EA norm}
\|A - \E A\|_{\infty \to 1} = \max_{x,y \in \{-1,1\}^n} \sum_{i,j=1}^n (a_{ij} - \E a_{ij}) x_i y_j.
\end{equation}
Let us fix $x,y \in \{-1,1\}^n$.
Using the symmetry of $A - \E A$, the fact that diagonal entries of $A - \E A$ vanish and collecting the identical terms, 
we can express the sum in \eqref{eq: A-EA norm} as a 
sum of independent random variables
$$
\sum_{i < j} X_{ij},
\quad \text{where} \quad
X_{ij} =   2(a_{ij} - \E a_{ij}) x_i y_j.
$$
To control the sum $\sum_{i < j} X_{ij}$ we can use Bernstein's inequality, Theorem~\ref{thm: bernstein}.
There are $N = \frac{n(n-1)}{2}$ terms in this sum. Since $|x_i| = |y_i| = 1$ for all $i$, 
the average variance $\s^2$ of all terms $X_{ij}$ is at most $2^2$ times the average variance of all 
$a_{ij}$, which is $\pbar$. In other words, $\s^2 \le 4 \pbar$. 
Furthermore, $|X_{ij}| \le 2 |a_{ij} - \E a_{ij}| \le 2$ since $0 \le a_{ij} \le 1$ by assumption. 
Hence $M \le 2$. It follows that 
\begin{equation}         \label{eq: prob prelim}
\Pr{ \frac{1}{N} \sum_{i < j} X_{ij} > t } 
\le \exp \left( - \frac{N t^2/2}{4 \pbar + 2t/3} \right).
\end{equation}
Let us substitute $t = 6 \, (\pbar/n)^{1/2}$ here. 
Rearranging the terms and using that $N = \frac{n(n-1)}{2}$ and $\pbar > 9/n$ (so that $t < 2 \pbar$), 
we conclude that the probability in \eqref{eq: prob prelim} is bounded by $\exp(-3(n -1))$. 

Summarizing, we have proved that for every ${x,y \in \{-1,1\}^n}$
$$
\Pr{ \frac{2}{n(n-1)} \sum_{i,j=1}^n (a_{ij} - \E a_{ij}) x_i y_j 
  > 6 \Big( \frac{\pbar}{n} \Big)^{1/2} }
\le e^{-3(n-1)}.
$$
Taking a union bound over all $2^{2n}$ pairs $(x,y)$, we conclude that 
\begin{align*}
\Pr{ \max_{x,y \in \{-1,1\}^n} \frac{2}{n(n-1)} \sum_{i,j=1}^n (a_{ij} - \E a_{ij}) x_i y_j 
  > 6 \Big( \frac{\pbar}{n} \Big)^{1/2} }
&\le 2^{2n} \cdot e^{-3(n-1)} \\
&\le e^3 \cdot 5^{-n}.
\end{align*}
Rearranging the terms and using the definition \eqref{eq: A-EA norm} of the $\ell_\infty \to \ell_1$ norm, 
we conclude the proof of Lemma~\ref{lem: deviation}.
\end{proof}

\begin{remark}[The sum of entries]				\label{rem: sum of entries}
  Note that by definition, the quantity $\big| \sum_{i,j=1}^n (a_{ij} - \E a_{ij}) \big|$
  is bounded by $\|A - \E A\|_{\infty \to 1}$, and thus it can be controlled by Lemma~\ref{lem: deviation}. 
  Alternatively, a bound on this quantity follows directly from the last line 
  of the proof of Lemma~\ref{lem: deviation}. 
  For a future reference, we express it in the following way:
  \[
  \frac{2}{n(n-1)} \Big| \sum_{i < j} (a_{ij} - \E a_{ij}) \Big | \le 3 \, \pbar^{1/2} n^{-1/2}.
  \] 
\end{remark}

\section{Stochastic block model: proof of Theorem~\ref{thm: community detection}}		\label{s: community detection}

So far our discussion has been general, and the results could be applied to 
a variety of semidefinite programs on random graphs. 
In this section, we specialize to the community detection problem 
considered in Theorem~\ref{thm: community detection}.
Thus we are going to analyze the optimization problem \eqref{eq: SDP lambda},
where $A$ is the adjacency matrix of a random graph distributed according 
to the classical stochastic block model $G(n,p,q)$. 

As we already noticed, this is a particular case of the class of problems \eqref{eq: SDP general} 
that we analyzed in Section~\ref{s: SDP on Grothendieck}. In our case,
$$
B := A - \l E_n
$$ 
with $\lambda$ defined in \eqref{eq: lambda}, and the feasible set is
$$
\Mopt := \Mgr^+ = \left\{ Z :\; Z \succeq 0, \; \diag(Z) \preceq \one_n \right\}.
$$

\subsection{The maximizer of the reference objective function}

In order to successfully apply Lemma~\ref{lem: almost max}, we will now choose 
a reference matrix $R$ so that it is close to (but also conveniently simpler than) the expectation of $B$.
To do so, we can assume without loss of generality that  
$\CC_1 = \{1,\ldots, n/2\}$ and $\CC_2 = \{ n/2+1,\ldots,n\}$. 
Then we define $R$ as a block matrix
\begin{equation}         \label{eq: R}
R = 
\frac{p-q}{2}
\begin{bmatrix}
  \phantom{-}E_{n/2} & -E_{n/2} \\
  -E_{n/2} & \phantom{-}E_{n/2}
\end{bmatrix}
\end{equation}
where as usual $E_{n/2}$ denotes the $n/2 \times n/2$ matrix whose all entries equal $1$.

Let us compute the expected value $\E B = \E A - (\E \l) E_n$ and compare it to $R$.
To do so, note that the expected value of $A$ has the form
\[
\E A = 
\begin{bmatrix}
  p E_{n/2} & q E_{n/2} \\
  q E_{n/2} & p E_{n/2}
\end{bmatrix} +(1-p) I_n.
\]
(The contribution of the identity matrix $I_n$ is required here since the diagonal entries of $A$ 
and thus of $\E A$ equal $1$ due to the self-loops.) 
Furthermore, the definition of $\l$ in \eqref{eq: lambda} easily implies that 
\begin{equation}	\label{eq: expected lambda}
\E \l = \frac{1}{n(n-1)} \sum_{i \ne j} \E a_{ij} = \frac{p+q}{2} \frac{n^2}{n(n-1)} - \frac{p}{n-1} = \frac{p+q}{2} - \frac{p-q}{n-1}.
\end{equation}
Thus
\begin{equation}         \label{eq: Abar Bbar}
\E B = \E A - (\E \l) E_n = R + (1-p) \one_n - \frac{p-q}{n-1} E_n.
\end{equation}
In the near future we will think of $R$ as the leading term and other two terms as being negligible,
so \eqref{eq: Abar Bbar} intuitively states that $R \approx \E B$. We save this 
fact for later. 

\medskip

Using the simple form of $R$, we can easily determine the form of the solution $Z_R$ 
of the reference problem \eqref{eq: SDP reference}.

\begin{lemma}[The maximizer of the reference objective function]				\label{lem: Zbar}
  We have
  $$
  Z_R
  := \arg \max_{Z \in \Mopt} \ip{R}{Z}
  =
  \begin{bmatrix}
    \phantom{-}E_{n/2} & -E_{n/2} \\
    -E_{n/2} & \phantom{-}E_{n/2}
  \end{bmatrix}.
  $$    
\end{lemma}

\begin{proof}
Let us first evaluate the maximizer of $\ip{R}{Z}$ on the larger set $[-1,1]^{n \times n}$, 
which contains the feasible set $\Mopt$ according to \eqref{eq: Mgr+}.
Taking into account the form of $R$ in \eqref{eq: R},  
one can quickly check that the maximizer of $\ip{R}{Z}$ 
on $[-1,1]^{n \times n}$ is $Z_R$. Since $Z_R$ belongs to the smaller
set $\Mopt$, it must be the maximizer on that set as well.
\end{proof}

\subsection{Bounding the error}

We are going to conclude from Lemma~\ref{lem: deviation} and Lemma~\ref{lem: almost max} that the maximizer of the actual objective function, 
$$
\Zhat = \arg \max_{Z \in \Mopt} \ip{B}{Z},
$$
must be close to $Z_R$, the maximizer of the reference objective function.

\begin{lemma}[Maximizers of random and reference functions are close]		\label{lem: Zhat-Zbar}
  Assume that $\pbar$ satisfies \eqref{eq: pbar}.
  Then, with probability at least $1-  e^3 5^{-n}$, we have 
  $$
  \|\Zhat - Z_R\|_2^2 \le \frac{116 \, \pbar^{1/2} n^{3/2}}{p-q}.
  $$
\end{lemma}

\begin{proof}
We expand
\begin{equation}							\label{eq: Zhat-Zbar expand}
\|\Zhat - Z_R\|_2^2 = \|\Zhat\|_2^2 + \|Z_R\|_2^2 - 2 \ip{\Zhat}{Z_R}
\end{equation}
and control the three terms separately. 

Note that $\|\Zhat\|_2^2 \le n^2$ since $\Zhat \in \Mopt \subset [-1,1]^{n \times n}$ 
according to \eqref{eq: Mgr+}. Next, we have 
$\|Z_R\|_2^2 = n^2$ by Lemma~\ref{lem: Zbar}. Thus 
\begin{equation}         \label{eq: Zhat smaller}
\|\Zhat\|_2^2 \le \|Z_R\|_2^2.
\end{equation}
Finally, we use Lemma~\ref{lem: almost max} to control the cross term in \eqref{eq: Zhat-Zbar expand}. 
To do this, notice that \eqref{eq: R} and Lemma~\ref{lem: Zbar} imply that  
$R = \frac{p-q}{2} \cdot Z_R$. 
Then, by homogeneity, the conclusion of Lemma~\ref{lem: almost max} implies that
\begin{equation}							\label{eq: ZR Zhat begin}
\ip{Z_R}{\Zhat} 
\ge \ip{Z_R}{Z_R} -\frac{4 K_\Gr}{p-q} \| R-B \|_{\infty \to 1}.
\end{equation}
To bound the norm of $R-B$, let us express this matrix as 
\begin{equation}							\label{eq: B-R}
B - R 
= (B - \E B) + (\E B - R) 
= (A - \E A) - (\l - \E \l) E_n + (\E B - R)
\end{equation}
and bound each of the three terms separately.
According to Lemma~\ref{lem: deviation} and Remark~\ref{rem: sum of entries}, 
we obtain that with probability larger than $1 - e^3 5^{-n}$, 
\[
\|A - \E A\|_{\infty \to 1} \le 3 \pbar^{1/2} n^{3/2} 
\quad \text{and} \quad
|\l - \E \l | \le 3 \pbar^{1/2} n^{-1/2}.
\]
Moreover, according to \eqref{eq: Abar Bbar},
$$
\E B - R = (1-p) \one_n - \frac{p-q}{n-1} E_n.
$$
Substituting these bounds into \eqref{eq: B-R} and using triangle inequality 
along with the facts that $\|E_n\|_{\infty \to 1} = n^2$, $\|\one_n\|_{\infty \to 1} = n$, we obtain
$$
\| B - R \|_{\infty \to 1} \le 6 \pbar^{1/2} n^{3/2} + (1-p)n + \frac{(p-q) n^2}{n-1}.
$$
Since $\pbar \ge 9/n$, one can check that each of the last two terms is bounded by $\pbar^{1/2} n^{3/2}$.
Thus we obtain
$\| B - R \|_{\infty \to 1} \le 8 \pbar^{1/2} n^{3/2}$.
Substituting into \eqref{eq: ZR Zhat begin}, we conclude that
\[
\ip{Z_R}{\Zhat} 
\ge \ip{Z_R}{Z_R} - 8  \pbar^{1/2} n^{3/2} \cdot \frac{4 K_\Gr}{p-q}.
\]
Recalling from \eqref{eq: Grothendieck constant} that Grothendieck's constant $K_\Gr$ 
is bounded by $1.783$, we can replace $8 \cdot 4 K_\Gr$ by $58$ in this bound. 
Substituting it and \eqref{eq: Zhat smaller} and into \eqref{eq: Zhat-Zbar expand}, 
we conclude that 
$$
\|\Zhat - Z_R\|_2^2 \le 2 \|Z_R\|_2^2 - 2 \ip{\Zhat}{Z_R} 
\le \frac{116 \, \pbar^{1/2} n^{3/2}}{p-q}.
$$
The proof of Lemma~\ref{lem: Zhat-Zbar} is complete.
\end{proof}

\medskip

\begin{proof}[Proof of Theorem~\ref{thm: community detection}]
The conclusion of the theorem will quickly follow from Lemma~\ref{lem: Zhat-Zbar}. 
Let us check the lemma's assumption \eqref{eq: pbar} on $\pbar$.
A quick computation yields
\begin{equation}	\label{eq:estimation pbar}
\pbar = \frac{2}{n(n-1)} \sum_{i <j} \Var(a_{ij})
= \frac{p(1-p)(n-2)}{2(n-1)} + \frac{q(1-q)n}{2(n-1)}. 
\end{equation}
Since $p(1-p) \le 1/4$, we get 
\[
\pbar \ge \frac{1}{2} \max \left\{ p(1-p), q(1-q) \right\} - \frac{1}{8(n-1)} > \frac{9}{n}
\]
where the last inequality follows from an assumption of Theorem~\ref{thm: community detection}.
Thus the assumption \eqref{eq: pbar} holds, and we can apply Lemma~\ref{lem: Zhat-Zbar}. 
It states that 
\begin{equation}         \label{eq: Zhat-xxtran prelim}
\|\Zhat - Z_R\|_2^2 
\le \frac{116 \, \pbar^{1/2} n^{3/2}}{p-q}
\end{equation}
with probability at least $1-  e^3 5^{-n}$.
From \eqref{eq:estimation pbar}, it is not difficult to see that 
$\pbar \le \frac{p+q}{2}$.
Substituting this into \eqref{eq: Zhat-xxtran prelim} and expressing $p = a/n$ and $q=b/n$, we conclude that
$$
\|\Zhat - Z_R\|_2^2 
\le \frac{116 \sqrt{(a+b)/2}}{a-b} \cdot n^2. 
$$
Rearranging the terms, we can see that this expression is bounded by $\e n^2$ if 
\[
(a-b)^2 \ge 7 \cdot 10^3 \e^{-2} (a+b).
\]
But this inequality follows from the assumption \eqref{eq: ab}. 

It remains to recall that according to Lemma~\ref{lem: Zbar}, we have
$Z_R = \xbar \xbar^\tran$ where $\xbar = [\onevector_{n/2} \; -\onevector_{n/2}] \in \R^n$
is the community membership vector defined in \eqref{eq: xbar}. 
Theorem~\ref{thm: community detection} is proved.
\end{proof}

\begin{proof}[Proof of Corollary \ref{cor: vector recovery}.]
The result follows from Davis-Kahan Theorem \cite{DK} 
about the stability of the eigenvectors under matrix perturbations.  
The largest eigenvalue of $ \xbar \xbar^\tran$ is $n$ while all the others are 0, 
so the spectral gap equals $n$. 
Expressing $\Zhat = (\Zhat - \xbar \xbar^\tran )+ \xbar \xbar^\tran$ 
and using that $\| \Zhat - \xbar \xbar^\tran \|_2 \le \sqrt \e n$, 
we obtain from Davis-Kahan's theorem  (see for example Corollary 3 in \cite{Vu}) that 
\[
\|\widehat{v} - \bar{v} \|_2 = 2 | \sin (\theta/2) | \le C \sqrt{\e}.
\] 
Here $\hat{v}$ and $\bar{v}$ denote the unit-norm eigenvectors associated to the largest eigenvalues of 
$\Zhat$ and $\xbar \xbar^\tran$ respectively, 
and $\theta \in [0, \pi/2]$ is the angle between these two vectors. By definition, $\xhat = \sqrt n \widehat{v}$
and $\xbar = \sqrt n \bar{v}$.
This concludes the proof.
\end{proof}

\section{General stochastic block model: proof of Theorem~\ref{thm: community detection general}}	\label{s: community detection general}

In this section we focus on the community detection problem for the 
general stochastic block-model considered in Theorem~\ref{thm: community detection general}.
The semidefinite program \eqref{eq: SDP sum fixed} is a particular case of the 
class of problems \eqref{eq: SDP general} that we analyzed in Section~\ref{s: SDP on Grothendieck}.
In our case, we set $B:=A$, choose the reference matrix to be 
$$
R := \bar{A} = \E A,
$$ 
and consider the feasible set 
$$
\Mopt := \Big\{ Z \succeq 0, \; Z \ge 0, \; \diag(Z) \preceq \one_n, 
  \; \sum_{i,j=1}^n Z_{ij} = \lambda \Big\}.
$$
Then $\Mopt$ is a subset of the Grothendieck's set $\Mgr^+$ defined in \eqref{eq: Mgr+}. 
Using \eqref{eq: Mgr+}, we see that
\begin{equation}							\label{eq: Mopt subset sum fixed}
\Mopt \subset \Big\{ Z :\; 0 \le Z_{ij} \le 1 \text{ for all } i,j; \; \; \sum_{i,j=1}^n Z_{ij} = \lambda \Big\}.
\end{equation}

\subsection{The maximizer of the expected objective function}

Unlike before, the reference matrix 
$R = \bar{A} = \E A = (p_{ij})_{i,j=1}^n$ is not necessarily 
a block matrix like in \eqref{eq: R} since the edge probabilities $p_{ij}$ 
may be different for all $i<j$.
However, we will observe that the solution $Z_R$ of the reference problem \eqref{eq: SDP reference} 
is a block matrix, and it is in fact the community membership matrix $\bar{Z}$ defined in \eqref{eq: Zbar membership}.

\begin{lemma}[The maximizer of the expected objective function]				\label{lem: Zbar sum fixed}
  We have 
  \begin{equation}							\label{eq: Zbar sum fixed}
  Z_R := \arg \max_{Z \in \Mopt} \ip{\bar{A}}{Z} = \bar{Z}.
  \end{equation}
\end{lemma}

\begin{proof}
Let us first compute the maximizer on the larger set $\Mopt'$, which contains 
the feasible set $\Mopt$ according to \eqref{eq: Mopt subset sum fixed}.
The maximum of 
the linear form $\ip{\bar{A}}{Z}$ on the convex set $\Mopt'$ is attained
at an extreme point. These extreme points are $0/1$ matrices
with $\l$ ones. Thus the maximizer of $\ip{\bar{A}}{Z}$ 
has the ones at the locations of the $\l$ largest entries of $\bar{A}$.

From the definition of the general stochastic block model we can recall that
$\bar{A} = (p_{ij})$ has two types of entries. The entries larger than $p$ 
form the community blocks $\CC_k \times \CC_k$, $k=1,\ldots,K$.
The number of such large entries is the same as the number of ones in the 
community membership matrix $\bar{Z}$, which in turn equals 
$\lambda$ by the choice we made in Theorem~\ref{thm: community detection general}.
All other entries of $\bar{A}$ are smaller than $q$.
Thus the $\l$ largest entries of $\bar{A}$ form the community blocks 
$\CC_k \times \CC_k$, $k=1,\ldots,K$.

Summarizing, we have shown that the maximizer of $\ip{\bar{A}}{Z}$ on the set 
$\Mopt'$ is a $0/1$ matrix 
with ones forming the community blocks $\CC_k \times \CC_k$, $k=1,\ldots,K$.
Thus the maximizer is the community membership matrix $\bar{Z}$ from \eqref{eq: Zbar membership}.
Since $\bar{Z}$ belongs to the smaller set $\Mopt$, 
it must be the maximizer on that set as well.
\end{proof}

\subsection{Bounding the error}

We are going to conclude from Lemma~\ref{lem: deviation} and 
Lemma~\ref{lem: almost max} that the maximizer of the actual objective function, 
$$
\Zhat = \arg \max_{Z \in \Mopt} \ip{A}{Z},
$$
must be close to $\bar{Z}$, the maximizer of the reference objective function.
We will first show that {\em the reference objective function 
$\ip{\bar{A}}{Z}$ distinguishes points near its maximizer $\bar{Z}$}.

\begin{lemma}[Expected objective function distinguishes points]		\label{lem: distinguishes points}
  Every $Z \in \Mopt$ satisfies
  \begin{equation}							\label{eq: distinguishes points}
  \ip{\bar{A}}{\bar{Z} - Z} \ge \frac{p-q}{2} \, \|\bar{Z} - Z\|_1.
  \end{equation}
\end{lemma}

\begin{proof}
We will prove that the conclusion holds for every $Z$ in the larger set $\Mopt'$, 
which contains the feasible set $\Mopt$ according to \eqref{eq: Mopt subset sum fixed}.
Expanding the inner product, we can represent it as
$$
\ip{\bar{A}}{\bar{Z} - Z} = \sum_{i,j=1}^n p_{ij} (\bar{Z}-Z)_{ij}
= \sum_{(i,j) \in \In} p_{ij} (\bar{Z}-Z)_{ij} - \sum_{(i,j) \in \Out} p_{ij} (Z - \bar{Z})_{ij}
$$
where $\In$ and $\Out$ denote the set of edges that run within and across the communities, 
respectively. Formally,
$\In = \cup_{k=1}^K (\CC_k \times \CC_k)$ and 
$\Out =\{1,\ldots, n\}^2 \setminus \In$.

For the edges $(i,j) \in \In$, we have $p_{ij} \ge p$ and $(\bar{Z}-Z)_{ij} \ge 0$ since 
$\bar{Z}_{ij}=1$ and $Z_{ij} \le 1$.
Similarly, for the edges $(i,j) \in \Out$, we have $p_{ij} \le q$ and $(Z - \bar{Z})_{ij} \ge 0$ since 
$\bar{Z}_{ij}=0$ and $Z_{ij} \ge 0$. It follows that 
\begin{equation}							\label{eq: pSin-qSout}
\ip{\bar{A}}{\bar{Z} - Z} 
\ge p S_{\In} - q S_{\Out}
\end{equation}
where 
$$
S_{\In} = \sum_{(i,j) \in \In} (\bar{Z}-Z)_{ij} 
\quad \text{and} \quad 
S_{\Out} = \sum_{(i,j) \in \Out} (Z - \bar{Z})_{ij}.
$$
Since both $\bar{Z}$ and $Z$ belong to $\Mopt$, the sum of all entries of both 
these matrices is the same ($n^2/2$), so we have 
\begin{equation}							\label{eq: Sin-Sout}
S_{\In} - S_{\Out} 
= \sum_{i,j=1}^n \bar{Z}_{ij} - \sum_{i,j=1}^n Z_{ij}
=0.
\end{equation}
On the other hand, as we already noticed, the terms in the sums that make 
$S_{\In}$ and $S_{\Out}$ are all non-negative. Therefore
\begin{equation}							\label{eq: Sin+Sout}
S_{\In} + S_{\Out} 
= \sum_{i,j=1}^n |(\bar{Z}-Z)_{ij}|
= \|\bar{Z}-Z\|_1. 
\end{equation}
Substituting \eqref{eq: Sin-Sout} and \eqref{eq: Sin+Sout} into \eqref{eq: pSin-qSout}, 
we obtain the conclusion \eqref{eq: distinguishes points}.
\end{proof}

Now we are ready to conclude that $\Zhat \approx \bar{Z}$.

\begin{lemma}[Maximizers of random and expected functions are close]		\label{lem: Zhat-Zbar sum fixed}
  Assume that $\pbar$ satisfies \eqref{eq: pbar}.
  With probability at least $1-  e^3 5^{-n}$, we have
  $$
  \|\Zhat - \bar{Z}\|_1 \le \frac{12 \,K_\Gr \, p_0^{1/2} n^{3/2}}{p-q}.
  $$
\end{lemma}

\begin{proof}
Using first Lemma~\ref{lem: distinguishes points}, 
Lemma~\ref{lem: almost max} (with $R=\bar{A}$ and $Z_R=\bar{Z}$ as before)
and then Lemma \ref{lem: deviation}, we obtain
$$
\|\Zhat - \bar{Z}\|_1 
\le \frac{2}{p-q} \, \ip{\bar{A}}{\bar{Z} - \Zhat} \le \frac{4 K_\Gr}{p-q} \| A - \bar A\|_{\infty \to 1}
\le \frac{12 K_\Gr}{p-q} \pbar^{1/2} n^{3/2}
$$
with probability at least $1-  e^3 5^{-n}$.
Lemma~\ref{lem: Zhat-Zbar sum fixed} is proved.
\end{proof}

\medskip

\begin{proof}[Proof of Theorem~\ref{thm: community detection general}]
The conclusion follows from Lemma~\ref{lem: Zhat-Zbar sum fixed}. 
Indeed, substituting $p=a/n$, $q=b/n$ and $\pbar = g/n$ and rearranging the terms, we obtain 
$$
\|\Zhat - \bar{Z}\|_1 \le \frac{12 K_\Gr g^{1/2}}{a-b} \cdot n^2 \le \frac{22 g^{1/2}}{a-b} \cdot n^2
$$
since we know form \eqref{eq: Grothendieck constant} that Grothendieck's constant $K_\Gr$ 
is bounded by $1.783$.
Rearranging the terms, we can see that this expression is bounded by $\e n^2$ if 
$(a-b)^2 \ge  484 \, \e^{-2} g$, which is our assumption \eqref{eq: ab general}. 
This proves the required bound for the $\|\cdot\|_1$ norm. 

Since for any sequence $\sum |b_{i,j}|^2 \le \max |b_{i,j}| \sum |b_{i,j}|$, we get 
$$
\|\Zhat - \bar{Z}\|_2^2
\le \|\Zhat - \bar{Z}\|_\infty \cdot \|\Zhat - \bar{Z}\|_1.
$$
As we noted in \eqref{eq: Mopt subset sum fixed}, all entries of $\Zhat$ and $\bar{Z}$ 
belong to $[0,1]$ hence $\|\Zhat - \bar{Z}\|_\infty \le 1$. The bound for the Frobenius norm follows and
Theorem~\ref{thm: community detection general} is proved.
\end{proof}

\section{The balanced planted partition model.}					\label{s: balanced planted partition}

In this section, we will show that a direct generalization of the semidefinite program \eqref{eq: SDP lambda} 
(instead a completely different program \eqref{eq: SDP sum fixed}) is capable of 
detecting multiple communities of equal sizes. 

Assume that a graph on $n$ vertices $\{1, \ldots, n\}$ is partitioned in $K$ communities $\CC_1, \ldots, \CC_K$
of equal sizes $s=n/K$. For each pair of distinct vertices, we draw an edge with probability $p$ if both vertices 
belong to the same community and with probability $q$ if not. 
This is known as the {\em balanced partition model}. 

Let $A$ be the adjacency matrix of the random graph drawn according to the balanced partition model. 
The community structure of such a network is not only captured by
the community membership matrix $\Zbar$ defined in \eqref{eq: Zbar membership} 
but also by the matrix $\Pbar$ defined as 
\begin{equation}							\label{eq: Zbar membership - K}
\bar{P}_{ij} = 
\begin{cases}
  \frac{K-1}{n} & \text{if $i$ and $j$ belong to the same community}; \\
  -\frac{1}{n}, & \text{if $i$ and $j$ belong to different communities}.
\end{cases}
\end{equation}
As we will see, $\bar{P}$ is an orthogonal projection with rank $K-1$.
We are going to estimate $\bar{P}$ using 
the following semidefinite optimization problem:
\begin{equation}			\label{eq: SDP lambda - K}
\begin{aligned}
  &\text{maximize } \ip{A}{Z} - \l \ip{E_n}{Z} \\
  &\text{subject to } Z \succeq 0,  \; \diag(Z) \preceq \one_n, \; \min_{i,j} Z_{ij} \ge -1/(K-1).
\end{aligned}
\end{equation}
For the value of $\lambda$, we choose the average degree of the graph as in \eqref{eq: lambda}, that is
\begin{equation}							\label{eq: lambda - K}
\l = \frac{2}{n(n-1)} \sum_{i < j} a_{ij}.
\end{equation}

This semidefinite program is a direct generalization of \eqref{eq: SDP lambda}. Indeed, when there are two communities, 
that is $K=2$, the constraint $\min_{i,j} Z_{ij} \ge -1$ can be dropped 
since it follows from $\diag(Z) \preceq \one_n$ \eqref{eq: Mgr+}).

\begin{theorem}[Community detection in the balanced partition model]	\label{thm: community detection - K}
  Let $\e \in (0,1/2)$, $K, s, n$ be integers with $n \ge 10^4 K / \e^2$ and $n = Ks$. 
  Let $A = (a_{i,j}) $ be the adjacency matrix of the random graph 
  drawn from the balanced partition model described above, with $K$ communities of equal sizes $s$.
  Let $p = \frac{a}{s} > q = \frac{b}{s}$ be such that 
  \begin{equation}	\label{eq:ab - K}
  (a-b)^2 \ge 50^2 \e^{-2}  (a+b(K-1)).
  \end{equation}
  Assume also that $\max\left\{ a(1-p), (K-1)b(1-q) \right\} \ge 10$.
Let $\Zhat$ be a solution of the semidefinite program \eqref{eq: SDP lambda - K} 
and $\Phat$ be the orthogonal projection onto the span of the eigenvectors 
associated to the $2K-3$ largest eigenvalues. 
Then, with probability at least $1-  e^3 5^{-n}$, we have
\[
\| \Phat - \Pbar \|_2^2 \le 8 \e \| \Pbar \|_2^2.
\]
\end{theorem}

\begin{remark}
  In the case of two communities, that is $K=2$, 
  Theorem~\ref{thm: community detection - K} yields the same conclusion as 
  Corollary \ref{cor: vector recovery}.
\end{remark}

\begin{remark}
It comes from the proof that if $\e \le 1/\sqrt{K-1}$, then the same conclusion holds true with $\Phat$ replaced by the orthogonal projection onto the span of the eigenvectors 
associated to the $K-1$ largest eigenvalues of $\Zhat$, thus making $\Pbar$ and $\Phat$ of the same rank.
\end{remark}

\begin{proof}
We build all the steps in one proof. 
Without loss of generality we assume that $\CC_i = \{(i-1)s +1, \ldots, is \}$ 
for all  for all $i = 1, \ldots, K$. Then $\Pbar$ is the block matrix that has the following form: 
\[
\bar{P} =
 \begin{bmatrix}
  \frac{K-1}{sK} E_{s}    & - \frac{1}{sK} E_{s}                      & \cdots                      &  - \frac{1}{sK} E_{s}  \\
  - \frac{1}{sK} E_{s}                    &  \frac{K-1}{sK}  E_{s}      &  - \frac{1}{sK} E_{s} & \cdots \\
  \vdots                                          & \vdots                                            & \vdots                         &
 \end{bmatrix}
\]
In other words, $\bar{P}$ has $K^2$ blocks of size $s \times s$. 
The diagonal blocks equal $ \frac{K-1}{sK} E_s$, the off-diagonal blocks equal $- \frac{1}{sK} E_{s}$, 
where as usual $E_{s}$ denotes the $s \times s$ matrix whose all entries equal $1$. 
Since $E_s^2 = s E_s$, it is easily seen that $\Pbar^2 = \Pbar$, 
so $\Pbar$ is an orthogonal projection of rank $K-1$. 
We define the reference matrix $R$ as follows:
\begin{equation}         \label{eq: R - K}
R := s(p-q) \Pbar = \frac{(K-1)(p-q)}{K}
\begin{bmatrix}
  \phantom{-}E_{s}    & - \frac{1}{K-1} E_{s}                      & \cdots                      &  - \frac{1}{K-1} E_{s}  \\
  - \frac{1}{K-1} E_{s}                    & \phantom{-}E_{s}      &  - \frac{1}{K-1} E_{s} & \cdots \\
  \vdots                                          & \vdots                                            & \vdots                         &
 \end{bmatrix}.
\end{equation}

Let us compute the expectated values of $\l$ and $A$. 
By definition \eqref{eq: lambda - K} of $\l$, we have
\[
\E \l = \left( \frac{1}{K} p + \frac{K-1}{K} q \right) + \frac{(K-1)(p-q)}{K(n-1)} = : \l_{pq} + \frac{(K-1)(p-q)}{K(n-1)}.
\]
Furthermore,
$$
\E A = \Abar + (1-p) \one_n, 
\quad \text{where} \quad
\Abar := 
\begin{bmatrix}
  p E_{s}    & q E_{s}                      & \cdots                      &  q E_{s}  \\
  q E_{s}                    & p E_{s}      &  q E_{s} & \cdots \\
  \vdots                                          & \vdots                                            & \vdots                         &
 \end{bmatrix}.
$$
(In the definition of $\bar{A}$, the diagonal blocks equal  $pE_s$ and all off diagonal blocks equal $qE_s$.) 
The difference between $\E A$ and $\Abar$ comes from the requirement that the diagonal entries of $A$ equal one. The matrix $R$ is chosen such that 
$$
\Abar - \l_{pq} E_n = R.
$$ 
One can thus check that 
$$
A - \l E_n - R  
= (A - \E A) - (\l - \E \l) E_n + (1-p) \one_n -  \frac{(K-1)(p-q)}{K(n-1)} E_n .
$$
We know from Lemma~\ref{lem: deviation} that if 
\[							
  \pbar := \frac{2}{n(n-1)} \sum_{i<j}^n \Var(a_{ij}) > \frac{9}{n}
\]
then 
with probability larger than $1 - e^3 5^{-n}$, 
\[
\|A - \E A\|_{\infty \to 1} \le 3 \pbar^{1/2} n^{3/2} 
\ \hbox{and} \
|\l - \E \l | \le 3 \pbar^{1/2} n^{-1/2}.
\]
Hence
\begin{align}	
\nonumber
\| A - \l E_n - R  \|_{\infty \to 1} &\le 6 \pbar^{1/2} n^{3/2} + (1-p) n + \frac{(K-1)(p-q)}{K(n-1)} n^2 
\\
\label{eq:norm difference}
&\le 14 \pbar^{1/2} n^{3/2}.
\end{align}


The second step of the argument consists of determining the form of the maximizer 
$$
Z_R := \arg \max_{Z \in \Mopt} \ip{R}{Z}
$$
where 
\[
\Mopt := \big \{ Z \succeq 0,  \; \diag(Z) = \one_n, \;  \min_{i,j} Z_{ij} \ge - 1/(K-1) \big\}.
\]
Note that all $Z \in \Mopt$ automatically satisfy $Z \in [-1/(K-1),1]^{n \times n}$ by \eqref{eq: Mgr+}.
As in Lemma \ref{lem: Zbar}, we can check that $Z_R$ is a multiple of $\Pbar$, more precisely
\begin{equation}	 \label{eq:def ZR}
  Z_R = \frac{sK}{K-1} \Pbar = 
  \begin{bmatrix}
  \phantom{-}E_{s}    & - \frac{1}{K-1} E_{s}                      & \cdots                      &  - \frac{1}{K-1} E_{s}  \\
  - \frac{1}{K-1} E_{s}                    & \phantom{-}E_{s}      &  - \frac{1}{K-1} E_{s} & \cdots \\
  \vdots                                          & \vdots                                            & \vdots                         &
 \end{bmatrix}.
\end{equation}
(Note that since $\Pbar$ is an orthogonal projection, we have $Z_R \succeq 0$. Hence $Z_R \in \Mopt$. The rest of the proof works as in Lemma \ref{lem: Zbar}.)

To bound the error, we establish as in Lemma \ref{lem: Zhat-Zbar} that the maximizer of the 
actual objective function, 
$$
\Zhat = \arg \max_{Z \in \Mopt} \ip{A - \l E_n}{Z},
$$
must be close to $Z_R$, the maximizer of the {\em reference} objective function.
Since $Z_R =  \frac{K}{(K-1)(p-q)} R$, we proceed identically using Lemma \ref{lem: almost max} and the bound for the $\ell_\infty \to \ell_1$ norm proved in \eqref{eq:norm difference} and we obtain
\begin{equation} 	\label{eq:main ineq}
\|\Zhat - Z_R\|_2^2 
\le \frac{50 \, K \,  \pbar^{1/2} n^{3/2}}{(K-1)(p-q)}
\end{equation}
under the condition that 
\[					
  \pbar := \frac{2}{n(n-1)} \sum_{i<j}^n \Var(a_{ij}) > \frac{9}{n}.
\]
In our case,
\[
\pbar = \frac{n}{n-1} \left(  \frac{1}{K} \, p(1-p) \left(1-\frac{1}{s} \right) +  \frac{K-1}{K} \,q (1-q) \right).
\]
We see that
\[
\pbar \ge \max\left\{ \frac{p(1-p)}{K}, \frac{(K-1)q(1-q)}{K} \right\} - \frac{1}{4(n-1)} > \frac{9}{n}
\]
if $\max\left\{ p(1-p), (K-1)q(1-q) \right\} \ge 10\, K / n$. Moreover,
\begin{align*}
\pbar - \left(\frac{p}{K} + \frac{(K-1) q}{K} \right)
& = \ \pbar - \left(\frac{p(1-p)}{K} + \frac{(K-1) q(1-q)}{K} \right) - \frac{p^2}{K} - \frac{(K-1)q^2}{K}
\\
& = \ \frac{K-1}{K(n-1)}\left( q(1-q) - p(1-p) \right) - \frac{p^2}{K} - \frac{(K-1)q^2}{K}
\\
& \le  \frac{(K-1)p^2}{K(n-1)} - \frac{p^2}{K}	= 	\frac{(K-n)p^2}{K(n-1)} \le 0
\end{align*}
since $q (1-q) < p$ and $K \le n$. 

We get that
\[
\frac{9}{n} \le \pbar \le \left(\frac{p}{K} + \frac{(K-1) q}{K} \right) = \l_{pq}.
\]
Recall that $a >1$ and $b<a$ are such that $p=a/s = aK/n$, $q=b/s = bK/n$. 
Together with \eqref{eq:main ineq}, this yields 
\[
\|\Zhat - Z_R\|_2^2 
\le 50 \, \frac{\sqrt{a + b(K-1)}}{a-b}  \frac{n^2}{K-1}.
\]
Observe also that $\| Z_R\|_2^2 = \frac{n^2}{K-1}$. 
With \eqref{eq:ab - K}, we have proved that 
\begin{equation}	\label{eq:final ineq}
\|\Zhat - Z_R\|_2^2 \le \e  \| Z_R\|_2^2 = \e \frac{n^2}{K-1}.
\end{equation}
For a symmetric $n \times n$ matrix $M$,  $\{\l_1(M), \ldots, \l_n(M)\}$ is the set of eigenvalues of $M$ arranged in decreasing order, that is $\l_1(M) \ge \ldots \ge \l_n(M)$. 
Recall that $Z_R = \frac{sK}{K-1} \Pbar$ is a multiple of an orthogonal projection of rank $(K-1)$ hence 
\begin{equation}	\label{eq:spectra}
\l_1(Z_R) = \ldots = \l_{K-1}(Z_R) = \frac{sK}{K-1}, \
\l_K(Z_R) = \ldots = \l_n(Z_R) =0.
\end{equation}
From Weyl's inequalities, see for example Theorem III.2.1 in \cite{Bhatia}, we have
\[
0  \le  \l_{2K-2}(\Zhat) \le \l_K (Z_R) + \l_{K-1} (\Zhat- Z_R) = \l_{K-1} (\Zhat- Z_R)
\]
since $\Zhat \succeq 0$ and $\l_K(Z_R)=0$. Hence $\l_1(\Zhat- Z_R) \ge \ldots \ge \l_{K-1} (\Zhat- Z_R) \ge 0$ and we deduce from \eqref{eq:final ineq} that
\[
(K-1) \l_{K-1}^2(\Zhat - Z_R) \le \sum_{i=1}^{K-1} \l_i^2(\Zhat- Z_R) \le \|\Zhat - Z_R\|_2^2 \le \e   \frac{n^2}{K-1} .
\]
Therefore
\begin{equation} 	\label{eq:spectra Zhat}
0 \le \l_n(\Zhat) \le \ldots \le \l_{2K-2}(\Zhat) \le \sqrt \e \frac{n}{K-1} = \sqrt \e \frac{sK}{K-1}.
\end{equation}
By definition, $\Phat$ is the orthogonal projection onto the span of the eigenvectors of $\Zhat$ associated to the eigenvalues $\{\l_1(\Zhat), \ldots, \l_{2K-3}(\Zhat)\}$ and $(\one_n - \Phat)$ is the orthogonal projection onto the span of the eigenvectors of $\Zhat$ associated to the eigenvalues $\{\l_{2K-2}(\Zhat), \ldots, \l_{n}(\Zhat)\}$. Let
\[
S_1 = \{\frac{sK}{K-1} \}
\quad
\text{and}
\quad
S_2 = [0, \sqrt \e \frac{sK}{K-1}].
\]
From \eqref{eq:spectra}, the projection onto the span of the eigenvectors of $Z_R$ whose eigenvalues belong to $S_1$ equals $\Pbar$, thanks to the observation \eqref{eq:def ZR}. From \eqref{eq:spectra Zhat}, the projection onto the span of the eigenvectors of $\Zhat$ whose eigenvalues belong to $S_2$ equal $(\one_n -\Phat)$. And the gap between $S_1$ and $S_2$ equals $(1-\sqrt \e)\frac{sK}{K-1}$.
We deduce from a well-known theorem due to Davis and Kahan \cite{DK}, see for example Theorem VII.3.1 in \cite{Bhatia}, that
\[
\| \Pbar (\one_n - \Phat)\|_2 \le \frac{\|Z_R - \Zhat||_2}{(1-\sqrt \e)\frac{sK}{K-1}}.
\]
Observe that 
\[
\|\Pbar - \Phat\|_2^2 = 2 \| \Pbar (\one_n - \Phat)\|_2^2
\
\text{and} 
\
\| \Pbar \|_2^2 = K-1
\]
and the conclusion of  Theorem \ref{thm: community detection - K} follows from \eqref{eq:final ineq}.
\end{proof}

\end{document}